\documentclass[12pt,reqno]{amsart}
\usepackage{amssymb}
\usepackage{amsmath}
\usepackage{amsthm}
\usepackage{eucal}
\usepackage{color}
\usepackage{amsfonts}
\usepackage{enumitem}
\usepackage[pdftex]{graphicx}
\usepackage[T1]{fontenc}

\newcommand{\R}{\mathbb R}

\DeclareMathOperator\supp{supp}

\DeclareMathOperator{\sech}{sech}
\newtheorem{theorem}{Theorem}[section]
\newtheorem{proposition}[theorem]{Proposition}
\newtheorem{remark}[theorem]{Remark}

\newtheorem{lemma}[theorem]{Lemma}

\newtheorem{definition}[theorem]{Definition}

\numberwithin{equation}{section}
\begin{document}
\title[Unique Continuation]{On unique continuation for non-local dispersive models}
\author{F.  Linares}
\address[F. Linares] {IMPA\\ Estrada Dona Castorina 110, Rio de Janeiro 22460-320, RJ Brazil}
\email{linares@impa.br}

\author{G. Ponce}
\address[G. Ponce]{Department  of Mathematics\\
University of California\\
Santa Barbara, CA 93106\\
USA.}
\email{ponce@math.ucsb.edu}

\keywords{Nonlinear dispersive equation,  propagation of regularity }
\subjclass{Primary: 35Q53. Secondary: 35B05}

\begin{abstract} 

We consider unique continuation properties of solutions to a family of evolution equations. Our interest is mainly on nonlinear non-local models. This class contains the Benjamin-Ono, the Intermediate Long Wave, the Camassa-Holm, the dispersion generalized Benjamin-Ono and non-local Schr\"odinger equations as well as their generalizations. We shall review, discuss, expand, and comment on several results. In addition, we shall state some open questions concerning these results and their techniques.

\end{abstract}

\maketitle

\section{Introduction}

In this manuscript we shall study unique continuation properties (UCP) of solutions to some time evolution equations. We are interested in two types of UCP: local  ones and asymptotic at infinity.

For  local ones we mean the following : if $u_1,\,u_2$ are solutions of the equation which agree in an open set
$\Omega$ (in the space-time space), then $u_1,\,u_2$  agree in their domain of definition.

Roughly, asymptotic at infinity  implies : if $u_1,\,u_2$ are solutions of the equation such that at two different times $t_1,\,t_2$
\begin{equation}
\label{nw}
||| \,u_1(\cdot,t_j)-u_2(\cdot,t_j)|||<\infty,\;\;\;\;j=1,2,
\end{equation}
then $u_1,\,u_2$ are equal in their domain of definition.
Here $|||\cdot|||$ may represent a \it weighted norm \rm or an asymptotic behavior at infinity. 

 For the linear equation one can fix $u_2\equiv 0$. In the nonlinear case, a weaker version of these UCP is obtained by assuming that $u_2\equiv 0$ is the second solution. In this case, for the nonlinear equation and $u_2\equiv 0$, the asymptotic at infinity UCP can be rephrased  in the question : what is the strongest possible decay at two different times of a nontrivial solution ?

The "norm" in \eqref{nw} may depend on the time $t_j,\,j=1,2$ and as we will see in some cases one needs three times to achieve the desired result.

The local UCP can be considered in solutions of the associated initial value problem (IVP), initial boundary value problem and mixed problems. In this work, we shall restrict ourselves to the IVP and to the initial periodic boundary value problem (IPBVP). Also, the asymptotic at infinity UCP will be only  considered here in solutions of the associated IVP.

It is clear that these UCP should be  examined in non-hyperbolic equations. Even in this class, non-hyperbolic models, these UCP may fail.
 For example, it was proved in \cite{RH} that the  (generalized Korteweg-de Vries) equation
\begin{equation}
\label{ggKdV}
\partial_tu+\partial_x^3(u^2)+\partial_x(u^2)=0,\;\;\;t,\,x\in \mathbb R,
\end{equation}
possesses  traveling waves solution $u(x,t)=\phi_c(x-ct)$ with $c>0$, called \it compacton\rm, of the form
\begin{equation}
\label{compacton}
\phi_c(\xi)=
\begin{cases}
\begin{aligned}
&\; \frac{4 c\,\cos^2(\xi/4)}{3},\;\;\;\;\;\;& |\xi|\leq 2\pi,\\
&\;0,\;\;\;\;\;\;\;&|\xi|>2\pi.
\end{aligned}\end{cases}
\end{equation}
We observe that $\phi_c(x-ct)$ is a classical solution of \eqref{ggKdV} with compact support. Thus, taking $u_1(x,t)=\phi_c(x-ct)$ and $u_2(x,t)=\phi_c(x-ct+4\pi)$ one gets a counterexample for our UCP's above.

Also, in \cite{Bar} it was proved that the porous medium equation
\begin{equation}
\label{porous}
\partial_tu=\partial_x^2(u^{1+m}),\;\;\;t>0,\,\,x\in \mathbb R,\;\;m>0,
\end{equation}
possesses non-negative compact support (generalized) solutions.

Although our main interest  here are non-local nonlinear dispersive problems, we shall first review the known UCP results for the Korteweg-de Vries equation. This will allow us to illustrate the difference with the non-local case. Next, we shall consider   the Benjamin-Ono equation (section 3),  the Intermediate Long Wave equation (section 4), the Camassa-Holm equation (section 4) and related models, the  dispersion generalized Benjamin-Ono equation (section 5), and  the general non-local Schr\"odinger equation (section 6).

We observe that the Korteweg-de Vries equation as well as  the equations  in sections 2-4  are completely integrable models, and the  models in sections 3-6  are non-local ones. Except for those in section 6, all the models are one dimensional.

As we shall see below the local UCP results for these nonlocal models are based in stationary arguments. Some of them are classic ones other have been proven recently and are interesting in their own right.
\section{The Korteweg-de Vries equation}

In this section we shall study the UCP on solutions to the Korteweg-de Vries (KdV) equation.  The KdV equation was first derived as a model of  propagation of waves on shallow water surfaces \cite{KdV}.  The KdV equation was the first equation to be solvable by means of the inverse scattering transform \cite{GGKM}. As it was already mentioned we shall use it to illustrate the difference with the non-local case. Thus,  we  shall examine the associated IVP, i.e.
\begin{equation}
\label{KdV}
\begin{cases}
\begin{aligned}
&\partial_tu+\partial_x^3u+u\partial_xu=0,\;\;\;t,\,x\in \mathbb R,\\
& u(x,0)=u_0(x),
\end{aligned}
\end{cases}
\end{equation}
and that for   its $k$-generalized form (k-gKdV)
\begin{equation}
\label{kgKdV}
\partial_t u + \partial_x^3 u + u^k\partial_x u=0,\;\;t,\,x\in \mathbb R,\;\;k\in\mathbb Z^+,
\end{equation}
for which only the cases $k=1,2$ are completely integrable.

The IVP and IPBVP (as well as other mixed problems) associated to the k-gKdV \eqref{kgKdV} have been extensively analyzed. In particular, their local and global well-posedness, the asymptotic behavior of their solutions, the stability of their special solutions (traveling waves and breathers) among other related topics have attracted great amount of attention, see \cite{LiPo}, \cite{Tao1}, \cite{KiVi} and references therein. In particular, we recall that the equation \eqref{kgKdV} has traveling wave solutions $u(x,t)=\phi_k(x-t)$ of the form
\begin{equation}
\label{twKdV}
\phi_k(x)=c_k\,\sech^{2/k}\Big(\frac{k\,x}{2}\Big)
\end{equation}
which belongs to the Schwartz class $\,\mathcal S(\mathbb R)$.

Concerning the local UCP for solution of \eqref{kgKdV} we observe : if  $\,u\,$  is a solution of the k-gKdV \eqref{kgKdV} in the domain $(x,t)\in \mathbb R\times [0,T]$ such  that $u(x,t)=0$ for all $(x,t)\in \Omega $ open set in $\mathbb R\times [0,T]$ with $[a,b]\times[t_1,t_2]\subset \Omega$, $ \,a<b$ and $0<t_1<t_2<T$, then the functions
\begin{equation}
\label{local1}
v_1(x,t)=
\begin{cases}
\begin{aligned}
&u(x,t),\;\;\;\;\;\;&(x,t)\in (-\infty,a)\times(t_1,t_2),\\
&\;0,\;\;\;\;\;\;&(x,t)\in [a,\infty)\times (t_1,t_2),
\end{aligned}
\end{cases}
\end{equation}
and
\begin{equation}
\label{local2}
v_2(x,t)=
\begin{cases}
\begin{aligned}
&\,0,\;\;\;\;\;\;\;&(x,t)\in (-\infty,b)\times(t_1,t_2),\\
&u(x,t),\;\;\;\;\;&(x,t)\in [b,\infty)\times (t_1,t_2),
\end{aligned}
\end{cases}
\end{equation}
are also solutions of the k-gKdV equation in the domain $\mathbb R\times(t_1,t_2)$. Similar result applies to the case of the difference $u_1-u_2$ of two solutions $u_1,\,u_2$ of the k-gKdV equation. 

Thus, in the k-gKdV (and in any local model) the local and asymptotic at infinity UCP are related.  This is not the case when one considers a 
 non-local model for which the argument  in \eqref{local1}-\eqref{local2} does not apply.

In the same regard, as we will see below, in a non-local model the hypothesis in the  local UCP can be replaced by the weaker one : 

if $u_1, u_2$ are solutions of the non-local equation such that exist a time $t_1$ and an open set $I$ in 
the space variable, such that
\begin{equation}
\label{con1}
\begin{cases}
\begin{aligned}
&\;\;\text{(i)}\hskip4pt\; \;\;\;\;u_1(x,t_1)=u_2(x,t_1),\,\,\,\,\,\;\;\;\;x\in I,\\
&\;\;\text{(ii)} \;\;\partial_tu_1(x,t_1)=\partial_tu_2(x,t_1),\;\;\;\;x\in I,
\end{aligned}
\end{cases}
\end{equation}
then $u_1\equiv u_2$.

This reflects the stationary character of the arguments used in  their proofs.

One has that for the KdV equation (or any local model) the condition  (i)  in \eqref{con1} combined with the equation implies that in  (ii). In fact,  under the hypothesis \eqref{con1} the local UCP  fails for the KdV 
(or any local model) by just taken $t_1=0$ and appropriate initial data $u_1(x,0)$ and $u_2(x,0)$.

Concerning the local UCP for the KdV equation we have the following result obtained in \cite{SaSc}.

\begin{theorem} [\cite{SaSc}]\label{87}
 If $u_1,\,u_2$ are "solutions" of the k-gKdV equation in $\mathbb R\times[0,T]$ such that  $u_1=u_2$ in 
an open  set  $\Omega \subset \mathbb R\times[0,T]$, then $u_1\equiv u_2$ in $\mathbb R\times[0,T]$.

\end{theorem}

 \begin{remark} \label{a1}\hskip10pt
 
  \begin{enumerate}
 \item Although the precise definition is not relevant for our discussion here, the type of solutions considered in \cite{SaSc} is quite general.

 \vspace{2mm}
 
\item The proof of Theorem \ref{87} in \cite{SaSc} is a consequence of a general Carleman estimate deduced there. It applies to a very general class of solutions to  local time evolution equation. The local character of the equation is essential in the argument of the proof.

\end{enumerate}
  \end{remark}

  Next, we consider the asymptotic at infinity UCP for the KdV equation and the k-gKdV equation. In this regard we have: 
  
  \begin{theorem}[\cite{Zh}]\label{12}
  
  If $u_1\in C(\mathbb R: H^4(\mathbb R))$  is a solution  of   the KdV equation \eqref{KdV}
  such that 
$$
u_1(x,t)=0,\;\;\;\;(x,t)\in(b,\infty)\times \{t_1,t_2\},\;\;\,b\in\mathbb R,\,t_1,\,t_2\in(0,T)
$$
(or in $\,(-\infty,b)\times \{t_1, t_2\}) ,\,b\in\mathbb R$), 
then $\;u_1\equiv 0$.
\end{theorem}
  
  \begin{remark} \label{A2}\hskip10pt
  \begin{enumerate}
 
 \item 
 The proof of Theorem \ref{12} is based on the inverse scattering method, using the fact that the KdV is a completely integrable model. It cannot be applied to the difference of any two solutions $u_1,\,u_2$, since it requires to have $u_2\equiv 0$. Moreover, it only applies to integrable cases.
  
\vspace{2mm}
\item
  Under the mixed assumption
$$
\hskip30pt u_1(x,t)=0,\,(x,t)\in(a,\infty)\times \{t_1\}\cup (-\infty, b)\times\{t_2\},\;a,\,b\in\mathbb R,
$$  
the  result $u_1\equiv 0$ is  \underline{only known} for the KdV equation ($k=1$ in \eqref{kgKdV}), in the case  $t_1<t_2$, see Remark \ref{A3} (2) below.

 \end{enumerate}
  \end{remark}

  The following result was established in \cite{Ta}:
  \begin{theorem}[\cite{Ta}]\label{13}
  
  If $u_1$ is the solution  of   the IVP associated to the KdV equation \eqref{KdV} with data $u_1(x,0)=u_0\in L^2(e^{\delta |x|^{1/2}}dx)$,$\,\delta>0$, then 
$u_1(x,t)$ becomes analytic in  $x$ 
for each  $\,t\neq 0$.
\end{theorem}

 \begin{remark} \label{A3} \hskip10pt
 
 \begin{enumerate}
 \item
 The proof of Theorem \ref{13} is based on the inverse scattering method.  In particular, it needs to assume that $u_2\equiv 0$. Similar, result is \underline{unknown} for  any nonlinearities $k\neq 1$ of the k-gKdV equation \eqref{kgKdV}.
 \vspace{2mm}
\item 
  In \cite{Ry} Theorem \ref{13} was obtained for $t>0$ under the decay assumption restricted  to $x>0$.
   \vspace{2mm}
 \item
 The fact that the decay of the data generates a gain of regularity on the solution was previously observed and studied in \cite{Ka}. 
 \end{enumerate}
\end{remark}

 The asymptotic at infinity UCP question  \eqref{nw} was answered in  \cite{EKPV07}:
  
 \begin{theorem}[\cite{EKPV07}]\label{14}
 There exists an universal constant
$$
 c_0=c_0(k)>0
 $$
such that if $u_1,\,u_2\in C([0,1]:H^4(\mathbb  R)\cap L^2(|x|^2dx))$ are
 solutions of  the k-gKdV equation \eqref{kgKdV}
satisfying 
\begin{equation}
\label{3/2}
 u_1(\cdot,0)-u_2(\cdot,0),\,\;\, u_1(\cdot,1)-u_2(\cdot,1)\in L^2(e^{c_0x_{+}^{3/2}}dx),
\end{equation}
then $u_1\equiv u_2$.
 \end{theorem}
 
 \medskip
 
 Above we have used the notation :
 $ x_{+}=max\{x;\,0\}$.
 
 \vspace{3mm}

\begin{remark}\hskip10pt

\begin{enumerate}
\item
The solution of the associated linear IVP to the k-gKdV in \eqref{kgKdV}
  \begin{equation}
 \begin{aligned}
 \begin{cases}
 &\partial_t v + \partial_x^3 v=0,\\
 &v(x,0)=v_0(x),
 \end{cases}
 \end{aligned}
\end{equation}
 is given by the group $\{V(t)\,:\,t\in \mathbb  R\}$ where 
 $$
 V(t)v_0(x)=\frac{1}{\root{3}\of{3t}}\,Ai\left(\frac{\cdot}{\root {3}\of{3t}}\right)\ast v_0(x), 
 $$
 with
 $$
 \;\;\;\;\;\;Ai(x)=c\,\int_{-\infty}^{\infty}\,e^{ ix\xi+i \xi^3 }\,d\xi\;\;\;\;\;\;\;\;\;\;\;\;\;\;\;\;\;\;\;\;\;\;\;\;\;\;\;
 $$
 is the Airy function which satisfies the decay estimate
 $$
 |Ai(x)|\leq c (1+x_{-})^{-1/4}\,e^{-c x_{+}^{3/2}},
 $$
which explains the exponent $3/2$.

\item
The previous remark suggests that the decay of the fundamental solution of  the associated linear IVP provides the appropriate weight for the "norm" in the asymptotic at infinity UCP, see
\eqref{nw}. In general, this is not the case. For example, the fundamental solution of the linear Schr\"odinger equation does not decay. In this case, the weight in \eqref{nw} is related to uncertainty principles for the Fourier transform. In fact, this weight may be different at time $t_1$ and $t_2$, see \cite{EKPV12}.

\item 

For previous results related to Theorem \ref{14} see \cite{Ro} and references therein.
\end{enumerate}
\end{remark}

Theorem \ref{14} is optimal as the following result shows:

\begin{theorem}[\cite{ILP13}]\label{15}
   
    Let $a_0>0$. For any given data
$$
u_0\in L^2(\mathbb R )\cap L^2(e^{a_0x_{+}^{3/2}}dx),
$$
the solution of the IVP for the KdV satisfies that for any $\,T>0$
$$
\sup_{t\in [0,T]}\,\int_{-\infty}^{\infty} e^{a(t)x_{+}^{3/2}}|u(x,t)|^2dx  \leq C^*,
$$
with
$$
C^*=C^*( a_0, \|u_0\|_2, \|e^{a_0x_{+}^{3/2}/2}u_0\|_2, T),
$$
and
$$
a(t)=\frac{a_0}{(1+27 a_0^2t/4)^{1/2}}.
$$

\end{theorem}

 \begin{remark}
 Theorem \ref{15} applies to solutions of the IVP associated to the  k-gKdV   equation \eqref{kgKdV} as well as the difference of two solutions of the k-gKdV   equation.
 \end{remark}
 
 We observe that the decay condition in Theorem \ref{14}   requires only one side decay at infinity condition. Thus, one may ask if   by assuming  a decay condition holding at booth of the real line sides this can be relaxed. For example, \underline{question} : if $u_1, \,u_2$ are appropriate solutions of the k-gKdV   equation \eqref{kgKdV} such that
 \begin{equation}
\label{1+}
 u_1(\cdot,0)-u_2(\cdot,0),\,\;\, u_1(\cdot,1)-u_2(\cdot,1)\in L^2(e^{a|x|}dx),\;\;\;\;\forall \,a>0,
\end{equation}
does this imply  that $u_1\equiv u_2$?

 To conclude this section we recall the so called "soliton resolution conjecture". This affirms, roughly speaking, that for generic initial data the corresponding solution of the k-gKdV will eventually resolve into a radiation component that disperses like a linear solution, plus a localized component that behaves like a soliton or multi-soliton solutions (or breathers). Assuming it, one expects that for  large time $t$ to have at most at decay  $u(x,t)\sim e^{ax}$ as $x\uparrow \infty$, since the solitons for the k-gKdV satisfy it, see \eqref{twKdV}. A precise result in this direction seems to be \underline{unknown}.

\section{The case of the  Benjamin-Ono equation}
 
 Next, we shall consider the Benjamin-Ono (BO) equation  
\begin{equation}
\label{BO}
\partial_t u - \mathcal  H\partial_x^2u +u\partial_x u = 0, \qquad t, x \in \mathbb R,
\end{equation}
where $\mathcal H$ denotes  the Hilbert transform
\begin{equation}
\label{HTR}
\mathcal H f(x)=\frac{1}{\pi} \,{\rm p.v.} (\frac{1}{x}\ast f)(x)=-i\,(sgn(\xi)\,\widehat {f}(\xi))^{\lor}(x).
\end{equation}
The BO equation was first deduced in \cite{Ben}  and \cite{On} as a model for long internal gravity waves in deep stratified fluids. Later, it was shown to be a completely integrable system \cite{AbFo}. We recall that the BO equation possesses traveling wave solutions (solitons) $u(x,t)=\phi(x-t)$ of the form
\begin{equation}
\label{twBO}
\phi_k(x)=\frac{4}{1+x^2},
\end{equation}
which is smooth but it exhibits a very mild decay in comparison with that of the k-gKdV \eqref{kgKdV} described in \eqref{twKdV}.

The well-posedness of the IVP and the IPBVP for the BO equation and its generalized k-form 
k-gBO equation, i.e.
\begin{equation}
\label{kgBO}
\partial_t u - \mathcal  H\partial_x^2u +u^k\partial_x u = 0, \qquad t, x\in\mathbb R,\;\;\;k\in\mathbb Z^+,
\end{equation}
has been broadly studied, we refer to \cite{Sa19} for a survey of these results.

Next, we shall introduce the following notation:
\begin{equation}
\begin{aligned}
\label{notZ}
&\,Z_{s,r}=H^s(\mathbb R)\cap L^2(|x|^{2r}dx),\;\;\;\;\;\;s,\,r\in\mathbb R,\\
\\
&\,\dot Z_{s,r}=Z_{s,r}\cap \{\,\widehat f(0)=0\},\;\;\;\;\;\;s,\,r\in\mathbb R.
\end{aligned}
\end{equation}

The known asymptotic at infinity UCP for the BO equation are embedded in the following well-posedness result  in weighted spaces $\,Z_{s,r}$:

\begin{theorem}  [\cite{FP10}-\cite{FLP12} after \cite{Io1}-\cite{Io2}] \label{BOres}
 
Let $\,u\in C([0,T] : H^5(\mathbb R))$ be the solution of the IVP for the BO  equation \eqref{BO} with data $u_0$.
\begin{enumerate}
\item
If $u_0\in Z_{5,r},\;\,r\in(0,5/2)$, then    $\,u\in C([0,T] : Z_{5,r})$.
\item
 If there exist  $\,0<t_1<t_2<T\,$ such that $\,u(\cdot,t_j) \in Z_{5,5/2},\;j=1,2$, then 
$\,\;\widehat u(0,t)=0,\;t\in[0,T]$.
\item
 If $u_0\in \dot Z_{5,r},\,r\in[5/2,7/2)$, then  $u\in C([0,T] : \dot Z_{5,r})$.
\item
 If there exist $t_1, t_2, t_3$ with  $0<t_1<t_2<t_3<T$
 such that  $\,u(\cdot,t_j) \in Z_{5,7/2},\,j=1,2,3,$ then $\,u\equiv 0$.
\item
 If $\,u_0\in \dot Z_{5,4}\;$ and $\;\int x\,u_0(x)dx\neq 0$, then

\begin{equation}
\label{t^*}\;\;\;\,\;\;\;u(\cdot,t^*) \in \dot Z_{5,4},\;\;\;\;\;\;\;\;t^*=- \frac{4\,\displaystyle\int xu_0(x)\,dx}{\|u_0\|^2_2}.
\end{equation}

\end{enumerate}
 \end{theorem}
 
 \begin{remark}\hskip10pt
 
 \begin{enumerate}
 \item 
In \cite{Io1}-\cite{Io2}  Theorem \ref{BOres} part (1) was proved for $r=1,2$, part (3) for $r=3$ and part (4) for $r=4$.
 \item
 
 The proof of Theorem \ref{BOres} is based on weighted energy estimates and relies on several inequalities for the Hilbert transform $\mathcal H$. Among them the so called $A_p$ condition introduced in \cite{Muc}, i.e. $w\in L^1_{loc}(\mathbb R)$ non-negative satisfies the $A_p,\,1<p<\infty$, condition  if 
 \begin{equation} \label{Ap}
 \sup_{Q\,\text{interval}}\,\Big(\frac{1}{|Q|}\int_Q w dx\Big)\Big(\frac{1}{|Q|}\int_Q w^{1-p'} dx\Big)^{p-1}=c_p(w)<\infty,
 \end{equation}
 where $1/p+1/p'=1$. In particular, $|x|^{\alpha}\in A_p$ if and only if $\alpha\in (-1,p-1)$. It was shown in \cite{HuMucWe} that the Hilbert transform  $\mathcal H$ is bounded in $L^p(wdx)$ if and only if $w$ satisfies \eqref{Ap}. Moreover, it was established in \cite{Pet} that in the case $p=2$ the operator norm is a multiple of $c_2(w)$ in \eqref{Ap}. The proof uses these results although not in their strongest versions.
 
 Also, the proof of Theorem \ref{BOres} relies on commutator estimates for the Hilbert transform $\mathcal H$, mainly the following : for $k,m \in \mathbb N\cup\{0\},\,k+m\geq 1$
 \begin{equation} \label{comm}
 \| \partial_x^k \big[\mathcal H;a\big]\partial_x^mf\|_P\leq c_{p,k,m}\|\partial_x^{k+m}a\|_{\infty} \|f\|_p.
 \end{equation}
 The case $k+m=1$ corresponds to the Calder\'on commutator estimate \cite{Ca}. The general case of \eqref{comm} was deduced in \cite{BCo}. For a different proof see \cite{DaGaPo}.
 \item
 The result in Theorem \ref{BOres} is due to the
connection between the dispersive 
relation (modeling by an operator with non-smooth symbol) and  the nonlinearity
of the BO equation \eqref{BO}. In particular, one can see that if $u_0\in \dot Z_{5,4}$
with $\int x u_0(x)dx\neq 0$,
then the solution $\,U(t)u_0(x)\,$ of the associated linear IVP 
$$
\partial_t u - \mathcal  H\partial_x^2u=0,\;\;\;\;\;u(x,0)=u_0(x),
$$
satisfies 
\begin{equation*}
\hskip30pt U(t)u_0(x)=c(e^{4\pi^2i t|\xi|\xi}\widehat {u_0}(\xi))^{\lor}\in L^2(|x|^{7-})-
L^2(|x|^7),\;\forall \,t\neq 0.
\end{equation*}
However, for the same data $u_0$ one has that the corresponding solution $u(x,t)$ of the BO equation 
\eqref{BO} satisfies
$$
u(\cdot,0),\,u(\cdot, t^*) \in L^2(|x|^8dx),
$$
and
$$
u(\cdot, t)\in
L^2(|x|^{7-})- L^2(|x|^7),\;\;\forall 
\,t\notin\{ 0, t^*\}.
$$
\item
 The value of $t^*$  in \eqref{t^*} can be motivated  by the identity
$$
\frac{d\,\,}{dt}\,\int_{-\infty}^{\infty}xu(x,t)dx=\frac{1}{2}
\|u(\cdot,t)\|_2^2=\frac{1}{2} \|u_0\|_2^2,
$$
(using the second conservation, i.e. the $L^2$-norm of the real solution)
which describes the time evolution of the first momentum of the solution. Hence, 
$$
\int_{-\infty}^{\infty}xu(x,t)dx= \int_{-\infty}^{\infty}xu_0(x)dx +\frac{t}{2}
\|u_0\|_2^2.
$$
So assuming that
\begin{equation}
\label{condition}
\int_{-\infty}^{\infty}x\,u_0(x)dx\neq 0,
\end{equation}
one looks for the times where the average of the first momentum of the solution
vanishes, i.e. for $t$ such that
$$
\int_0^t \int_{-\infty}^{\infty}x\,u(x, t')dx dt'= \int_0^t
\Big(\int_{-\infty}^{\infty}x\,u_0(x)dx +\frac{t'}{2} \|u_0\|_2^2\,\Big) dt'=0,
$$
which under the assumption \eqref{condition} has a unique solution $t=t^*$ given by
the formula in \eqref{t^*}.

\item
Theorem \ref{BOres} leaves \underline{open} the question of a UCP at infinity involving only two different times  by strengthening the hypothesis on the decay.
More precisely,  can one find $r>4$ such that if $u\in C([-T,T] : \dot Z_{s,7/2^{-}})$, $\;s\gg 1$, is a solution of the BO equation \eqref{BO}  satisfying that  there exist $t_1,\,t_2\in[-T,T]$, $t_1\neq t_2$ such that $u(t_j)\in \dot Z_{s,r},\;j=1,2$, then $u\equiv 0$?

In \cite{Flo} it was proved that this is not possible at least for $r<11/2$.

\item One sees  that the UCP in Theorem \ref{BOres} assumes that the second solution $u_2\equiv 0$. A result for asymptotic at infinity UCP involving two arbitrary solutions $u_1, \,u_2$ of the IVP for the BO equation remains \underline{open}.

\item The results (1)-(3) in Theorem \ref{BOres} extend to solutions of the IVP associated to the k-gBO equation in \eqref{kgBO}. However, part (4) and (5) (modified version) hold for solutions of the IVP associated to the k-gBO equation in \eqref{kgBO} when $k$ is odd.

 \end{enumerate}
 \end{remark}

 Next, we consider the local UCP for solutions of the IVP associated to the k-gBO equation \eqref{kgBO}. In this regard we have the following result:

\begin{theorem} [\cite{KPV20}]\label{BOloc}

 Let $$u_1,\,u_2\in  C([0,T]:H^s(\mathbb R)) \cap C^1((0,T):H^{s-2}(\mathbb R)),\;s>5/2,$$ be strong solutions of
k-gBO equation \eqref{kgBO}. If there exist an open interval $\,I \subset \mathbb R$  and $t_1\in(0,T)$ such that
$$
\begin{cases}
\begin{aligned}
&\;\;\;\;u_1(x,t_1)=u_2(x,t_1),\;\;\;\;\;\;\;x\in I,\\
&\;\partial_tu_1(x,t_1)=\partial_tu_2(x,t_1),\;\;\;\;x\in I,
\end{aligned}
\end{cases}
$$
 then,
$$
u_1(x,t)=u_2(x,t),\;\;\;\;(x,t)\in \mathbb R\times [0,T].
$$

 In particular, if $\,u_1(x,t)=u_2(x,t),\;\,(x,t)\in \Omega$ with $\Omega\subset \mathbb R\times [0,T]$ open, then $\,u_1\equiv u_2$. 
\end{theorem}

\begin{remark}\label{BO-rem}\hskip10pt

\begin{enumerate}

\item It is surprising that a period of more than 30 years separated the local UC result for the k-gKdV \eqref{kgKdV} (Theorem \ref{87}) and that for the k-gBO \eqref{kgBO} (Theorem \ref{BOloc}). Specially, since many results first established for the KdV equation have been studied in the BO equation, see \cite{LiPo}. Maybe, this is due to the difference in their proof arguments. More  concretely, the classical approach in \cite{SaSc}, \cite{Iz}, \cite{EKPV06}, \cite{EKPV07}, among others, is based on Carleman estimates which cannot be extended to non-local models. 
 For these ones, it seems that  simpler but very specific stationary arguments are needed. Some of them have just recently been established. 
\item 
Theorem \ref{BOloc}  extends to a pair of solutions $u_1,\,u_2$ of  the  Burgers-Hilbert (BH) equation 
$$  \partial_t u - \mathcal H   u + u^k\partial_x u = 0,\qquad (x,t) \in  \, \mathbb R\times \mathbb R,\;\;\;k\in \mathbb Z^+,$$
see \cite{BiHu}, \cite{HITW}.
 \item 
It also applies to solutions of the  IPBVP
 associated to the k-gBO equation \eqref{kgBO}. In this case the Hilbert transform is defined as 
$$
\mathcal H f(x)=\frac{1}{2\pi} \, {\rm p.v.} \,\int_0^{2\pi} f(t)\,\cot\Big(\frac{x-t}{2}\Big) \, dt.
$$

\end{enumerate}
\end{remark}
It will be clear from the proof of Theorem \ref{BOloc} sketched below that this can be generalized in the following form:

\begin{theorem} \label{BOlocg}

 Let $n\in \mathbb Z^+$ and $$u_1\in  C([0,T]:H^s(\mathbb R)) \cap C^1((0,T):H^{s-2}(\mathbb R)),\;s>(k+1)n+5/2,$$ be a strong solution of
k-gBO equation \eqref{kgBO}. If there exist  an open interval $\,I \subset \mathbb R$ and $t_1\in (0,T)$ such that
$$
\partial_tu_1(x,t_1)=Q_{(k+1)n-1}(x),\;\;\;\;\text{and}\;\;\;\;u(x,t_1)=P_n(x),\;\;\;\;\;x\in I,
$$
where $Q_{(k+1)n-1}$ and $P_n$ are polynomials with real coefficients of degree at most $(k+1)n-1$ and $n$, respectively,
 then
$$
Q_{(k+1)n-1}(x)\equiv P_n(x)\equiv u_1(x,t)\equiv  0,\;\;\;\;(x,t)\in \mathbb R\times [0,T].
$$

\end{theorem}

Next, we shall sketch the proof of Theorem \ref{BOloc}.

\begin{proof}  [Proof of Theorem \ref{BOloc}]

We need the following consequence of Schwarz reflection principle in complex analysis.

\begin{proposition}\label{prop1}  Let $\,I\subseteq \mathbb R\,$ be an open interval, $\,b\in(0,\infty]$  and
$$
D_b=\{z=x+iy\in \mathbb C:0<y<b\},\;\;L=\{x+i0\in \mathbb C:x\in I\}.
$$

 Let $\,F:D_b\cup L\to\mathbb C\,$ be continuous and $\,F \big|_{D_b}$ analytic.

 If $\,\,F \big|_{L}\equiv 0$, then $\,F\equiv 0$.
\end{proposition}

Using Proposition \ref{prop1} we get: 

\begin{lemma}\label{lemma1} Let  $\,f\in H^s(\mathbb R),\,s>1/2,\,$ be a real valued function. If there exists an open interval $I\subset \mathbb R$ such that 
$$
f(x)=\mathcal Hf(x)=0,\;\;\;\;\;\;\;\forall \,x\in I,
$$
then $f\equiv 0$.

\end{lemma}

By standard approximation, the same result holds assuming that  $f\in H^s(\mathbb R),\,s\in \mathbb R$. Thus, defining 
$$
F(x+iy)\equiv \int \,e^{i(x+iy)\xi}\,(\widehat{f+i\,\mathcal H\,f})(\xi)\,d\xi
$$
$$
=\int \,e^{i(x+iy)\xi}\,(1+\text{sgn}(\xi))\widehat{f}(\xi)d\xi=2\int_{ \xi\geq 0}e^{i(x+iy)\xi}\,\,\widehat{f}(\xi)d\xi$$
one has that 
$$
F(x+i\,0)=(f+i\,\mathcal H\,f)(x)
$$ 
is continuous and has an analytic extension $F(x+iy)$ on $y>0$.
Hence, using Proposition \ref{prop1} one proves  Lemma \ref{lemma1}.

 To conclude the proof of Theorem \ref{BOloc}  we define 
 $$\,w(x,t)=(u_1-u_2)(x,t),$$ 
 which satisfies the equation
$$
\partial_tw-\mathcal H \partial_x^2 w+\partial_xu_2\,w+u_1\,\partial_xw=0,\;(x,t)\in\mathbb R\times[0,T].
$$
By the hypothesis and the equation one has that
$$
w(x,t_1)=\mathcal H \partial^2_xw(x,t_1)=0,\;\;\;\;x\in I.
$$
Thus,  Lemma \ref{lemma1} yields the result.

\end{proof}

\section{The intermediate Long Wave equation }

The Intermediate Long Wave (ILW) equation can be written as 
\begin{equation}
\label{ILW}
\partial_t u - \mathcal L_\delta \partial_x^2  u + \frac1{\delta}\partial_x u  +  u\partial_x u = 0,\,\;(x,t) \in  \mathbb R^2,
\end{equation}
where $u=u(x,t)$  is a real-valued function, $\,\delta>0\,$ and
$$\mathcal L_\delta f(x) =-\frac1{2\delta}\,\rm{p.v.} \int \rm{coth}\it \left(\frac{\pi(x-y)}{2\delta}\right)f(y)dy.$$

Also, $\mathcal L_\delta$ is a multiplier operator with  symbol
$$\sigma(\partial_x\mathcal L_\delta)=\widehat{\partial_x\mathcal L_\delta} =2\pi \xi \,\rm{coth}\,(2\pi \delta \xi).$$

The ILW equation describes  long internal gravity waves in a stratified fluid with finite depth given by the parameter 
$\,\delta$, see \cite{KKD} and \cite{Jos}.  Moreover, it was shown in \cite{KAS81}-\cite{KAS82} that is formally a completely integrable model as the KdV, the CH and the BO equations, see also \cite{Iv}. The ILW equation has traveling waves solutions 
(solitons) $u_{\delta,c}(x,t)=\phi_{\delta}(x-ct), \,c>0$ of the form
\begin{equation}
\label{ilw-tv}
\phi_{\delta}(\xi)=\frac{ 2 a\,\sin(a\delta)}{\cosh(a\xi)+\cos(a\delta)},
\end{equation} 
 where $a\in(0,\pi/\delta)$ solves the equation  $\,a\delta \cot(a\delta)=1-c\delta$, see \cite{Jos}.        
            
            Regarding the IVP associated to the ILW equation  \eqref{ILW} it was proven in \cite{ABFS}   that solutions of the  ILW equation converge, as $\delta \to \infty$ (deep-water limit), to solutions of the IVP associated to the BO equation with the same initial data. Also, it was established in \cite{ABFS} that
if $u_{\delta}(x,t)$ denotes a solution of the ILW equation \eqref{ILW}, then
$$v_{\delta}(x,t)=\,\frac{3}{\delta} \,u_{\delta}\Big(x,\frac{3}{\delta} t\Big)$$
converges,  as $\delta\to 0$ (shallow-water limit), to a solution  of the KdV equation with the same initial data.

With respect to well-posedness of the IVP associated to ILW equation we refer to \cite{Sa19} and references therein.

We \underline{do not know} any result concerning  the asymptotic at infinity UCP for the ILW equation, even in the case where the second solution $u_2\equiv 0$.

In \cite{KPV20} the local UCP results for the BO equation were extended to the ILW equation :

\begin{theorem} [\cite{KPV20}]\label{ILWloc}

 Let $$u_1,\,u_2\in C([0,T]:H^s(\mathbb R)) \cap C^1((0,T):H^{s-2}(\mathbb R)),\;s>5/2,$$ be strong solutions of
ILW equation \eqref{ILW}.  If there exist an open interval $\,I \subset \mathbb R$  and $t_1\in(0,T)$ such that
$$
\begin{cases}
\begin{aligned}
&\;\;\;\;u_1(x,t_1)=u_2(x,t_1),\;\;\;\;\;\;\;x\in I,\\
&\;\partial_tu_1(x,t_1)=\partial_tu_2(x,t_1),\;\;\;\;x\in I,
\end{aligned}
\end{cases}
$$
 then
$$
u_1(x,t)=u_2(x,t),\;\;\;\;(x,t)\in \mathbb R\times [0,T].
$$

 In particular, if $\,u_1(x,t)=u_2(x,t),\;\,(x,t)\in \Omega$ with $\Omega\subset \mathbb R\times [0,T]$ open, then $\,u_1\equiv u_2$. 
  \end{theorem}

As in the case of  BO equation one has the following extension  of  the second part of Theorem \ref{ILWloc}:

\begin{theorem}\label{ILWlocg}

 Let $n\in \mathbb Z^+$ and 
 $$u_1\in C([0,T]:H^s(\mathbb R)) \cap C^1((0,T):H^{s-2}(\mathbb R)),\;s>2n+5/2$$ be a strong solutions of
ILW equation \eqref{ILW}. If there exist an open interval  $\,I \subset \mathbb R$ and $t_1\in (0,T)$ such that
$$
\partial_tu_1(x,t_1)=Q_{2n-1}(x),\;\;\;\text{and}\;\;\;\;u_1(x,t_1)=P_n(x),\;\;\;\;x\in I,
$$
where $Q_{2n-1}$ and $P_n$ are  polynomials with real coefficients of degree at most $2n-1$ and $n$ respectively, then
$$
Q_{2n-1}(x)\equiv P_n(x)\equiv u_1(x,t)\equiv 0,\;\;\;\;(x,t)\in \mathbb R\times [0,T].
$$
  \end{theorem}

 Next, we shall sketch the proof of Theorem \ref{ILWloc}.

\begin{proof}  [Proof of Theorem \ref{ILWloc}]

First, we shall have the following result related to that in Proposition \ref{prop1}:

\begin{proposition}

Let $f\in H^s(\mathbb R),\,s>3/2$ be a real valued function. If there exists an open set $I\subset \mathbb R$ such that 
$\,
f(x)=\mathcal L_{\delta}\partial_x f(x)=0,\,\;\,x\in I,
$
then $\;f\equiv 0$.
\end{proposition}

Define
$$
\;F(x)\equiv \partial_xf(x)+i \mathcal L_{\delta}\partial_x f(x),
$$
so
\begin{equation}
\begin{aligned}
\widehat{F}(\xi)&=\widehat{(\partial_xf+i\mathcal L_{\delta}\partial_xf)}(\xi)\\
&=\widehat{F}(\xi)=c\,\xi \Big(1+\frac{e^{2\pi\delta\xi}+e^{-2\pi\delta\xi}}{e^{2\pi\delta\xi}-e^{-2\pi\delta\xi}}\,\Big)\,\widehat{f}(\xi)\\
&=c \,\xi \,\frac{e^{4\pi\delta \xi}}{1-e^{4\pi \delta \xi}}\,\widehat{f}(\xi).
\end{aligned}
\end{equation}

Thus, $\,\widehat{F}\in L^1(\mathbb R)$ with appropriate exponential decay. Therefore $\,F(x)$
has an analytic extension 
$$
F(x+iy)=\int_{-\infty}^{\infty}\;e^{2\pi i \xi (x+iy)}\,\widehat{F}(\xi)\,d\xi
$$
to the strip
$
\;\;D_{2\delta}=\{z=x+iy\in \mathbb C\,:\,0<y<2\delta\}.
$
This completes the proof. 
 \end{proof}
\section{The Camassa-Holm equation and related models}

This section is mainly concerned with the  Camassa-Holm (CH) equation
\begin{equation}\label{CH}
\partial_tu+3 u \partial_xu-\partial_t \partial_x^2 u=2\partial_xu \partial_x^2u+ u\partial_x^3u, \hskip5pt \;t,x\in\R.
\end{equation}

The CH equation \eqref{CH} was first explicitly written in \cite{FF} a work on hereditary symmetries. Later, it was  explicitly derived as a physical model for shallow water waves  in \cite{CH}, where its solutions were also investigated. It has also appeared as a model in nonlinear dispersive waves in hyper-elastic rods \cite{Dai}. 

The CH equation \eqref{CH} has received extensive attention due to its remarkable properties, among them the fact that it is a bi-Hamiltonian completely integrable model  (see   \cite{CH}, \cite{CoMc}, \cite{Mc}, \cite{M}, \cite{Par} and references therein).

The CH equation has special  traveling waves solutions (solitons) called {\it peakons}
\begin{equation}\
\label{peakon}
u(x,t)=\phi(x-ct)=c\,e^{-|x-ct|},\;c>0.
\end{equation}
The multi-peakon solutions exhibit the "elastic" collision property that reflects their soliton character.

It is convenient to rewrite the CH equation as
\begin{equation}
\label{CH0}
\partial_t u+u\partial_xu+\partial_x(1-\partial_x^2)^{-1}\big(u^2+\frac{1}{2}\,\big(\partial_xu)^2\big)=0.
\end{equation}

The IVP as well as the IPBVP associated to the equation \eqref{CH0} has been considerably examined.   

In \cite{LO} and \cite{RB}  strong local well-posedness of the IVP associated to \eqref{CHk} was established in the classical Sobolev space
$
\,H^s(\R)=(1-\partial_x^2)^{-s/2}L^2(\R)$ with $s>3/2.
$

However, one observes that peakon solutions, described in \eqref{peakon} (case $k=0$), do not belong to these spaces.
In fact, one has that  for any $p\in [1,\infty)$
\begin{equation}\label{peakon-sobolev-p}
\phi(x)=e^{-|x|}\notin H^{1+1/p,p}(\R),
\end{equation}
where for $s\in \R$ and $p\in [1,\infty)$ 
\begin{equation}\label{Sobolev-p}
H^{s,p}(\R)=(1-\partial_x^2)^{-s/2}L^p(\R),
\end{equation}
with $\,H^{s,2}(\R)=H^s(\R)$, 
see \cite{LiPo}. 
However,
\begin{equation}\label{peakon-sobolev}
\phi(x)=e^{-|x|}\in W^{1,\infty}(\R),
\end{equation}
where $W^{1,\infty}(\R)$ denotes the space of Lipschitz functions.

In \cite{CoMo} it was proved that if  $u_0\in H^1(\R)$ with $u_0-\partial_x^2u_0\in \mathcal {M}^+(\R)$
(the set of positive Radon measures with bounded total variation),
then the  IVP for the CH equation \eqref{CH0} has a unique solution 
$$
u\in C([0,\infty):H^1(\R))\cap C^1((0,\infty):L^2(\R))
$$
 satisfying that
$ y(t)\equiv (1-\partial_x^2)u(\cdot,t)\in \mathcal {M}^+(\R)$ is uniformly bounded.

In \cite{XZ}  the existence of a $H^1$-global weak solution for the IVP for the CH equation \eqref{CH} for 
data $u_0\in H^1(\R)$ was established. 

In \cite{CoEs1} and \cite{CoEs2} (see also \cite{LO}) conditions on the data $u_0\in H^3(\R)$ were derived to guarantee that the
corresponding  local solution $u\in C([0,T]:H^3(\R))$ of the IVP associated to the CH equation \eqref{CH0} blows up in finite time, i.e. 
$$
\lim_{t\uparrow T}\|\partial_xu(\cdot,t)\|_{\infty}=\infty,
$$
which corresponds to the breaking of waves.

 Formally, one has that $H^1$-solutions of the CH equation \eqref{CH0}  satisfy the conservation law
$$
E(u)(t)=\int_{-\infty}^{\infty}\big(u^2+(\partial_xu)^2\big)(x,t)\,dx= E(u_0),
$$
so that the $H^1$-norm of the solutions constructed in  \cite{LiPoSi} (see Theorem \ref{CH3} below) remains invariant 
within the existence interval.
 
In \cite{BC} and \cite{BCZ}  the existence and uniqueness, 
 respectively,  of a $H^1$ global solution for the CH equation \eqref{CH} was established.

 For other well-posedness results we refer to  \cite{GHR1}, \cite{GHR2}  and references  therein.
\vskip.1in

Concerning asymptotic at infinity UCP for solutions of the CH equation, we recall two theorems obtained in \cite{HMPZ} :

\vskip.1in

\begin{theorem}[\cite{HMPZ}] \label{CH1}  Assume that for some $T>0$ and $s>3/2$,
\[
u\in C([0,T]: H^s(\R))\cap C^{1}((0,T):H^{s-1}(\R))
\]
is a strong solution of the IVP associated to the CH equation \eqref{CH0}. If for some $\alpha\in(1/2,1)$, $u_0(x)=u(x,0)$ satisfies  
\begin{equation}\label{h1}
|u_0(x)|= o(e^{-x})\;\;\;\;\text{and}\;\;\;\;|\partial_xu_0(x)|=  O(e^{-\alpha x}),\;\;\;\;\text{as}\;\;\;x\uparrow \infty,
\end{equation}
and there exists $t_1\in (0,T]$ such that 
\begin{equation}\label{h2}
|u(x,t_1)|= o(e^{-x}),\;\;\;\;\;\;\text{as}\;\;\;x\uparrow \infty,
\end{equation}
then $u\equiv 0$.
\end{theorem}

\begin{remark}\label{rem-dec-ch}\hskip10pt

\begin{enumerate} 

\item

The conclusion of Theorem \ref{CH1} still holds with the decay in \eqref{h1}-\eqref{h2} assumed for  $x\downarrow -\infty$.

\item Roughly, Theorem \ref{CH1} affirms that a non-trivial solution of the CH equation can decay faster at two times than peakon (soliton) solution \eqref{peakon}. This is not the case of the k-gKdV equation, k-gBO equation and the dgBO equation (section 6). In view of the soliton resolution conjecture this should be the case for very large values of the time $t$, see the comments at the end of the section 2.

\end{enumerate}
\end{remark}
The next result shows that Theorem \ref{CH1} is optimal:

\begin{theorem}[\cite{HMPZ}] \label{CH2}  Assume that for some $T>0$ and $s>3/2$,
\[
u\in C([0,T]: H^s(\R))\cap C^{1}((0,T):H^{s-1}(\R))
\]
is a strong solution of the IVP associated to the CH equation \eqref{CH0}. If 
for some $\theta\in(0,1)$,
$u_0(x)=u(x,0)$ satisfies 
\[
|u_0(x)|,\;\;\;\;|\partial_xu_0(x)|=  O(e^{-\theta x}),\;\;\;\;\text{as}\;\;\;x\uparrow \infty,
\]
then
\[
|u(x,t)|,\;\;\;\;|\partial_xu(x,t)|=  O(e^{-\theta x}),\;\;\;\;\text{as}\;\;\;x\uparrow \infty,
\]
uniformly in the time interval $[0,T]$.
\end{theorem}

\begin{remark} \hskip10pt

\begin{enumerate}

\item
We observe the UCP in Theorem \ref{CH1} is restricted to the case where $u_2\equiv 0$. A result of this kind for two arbitrary solutions $u_1,\,u_2$ of the CH equation \eqref{CH0} is \underline{unknown}. 

\item

In \cite{LiPoSi}  Theorem \ref{CH1} and Theorem \ref{CH2} were extended to the solutions of the CH equation \eqref{CH0} obtained in the following class which contains the peakons solutions \eqref{peakon}:

\end{enumerate}
\end{remark}

\vskip.1in

\begin{theorem}[\cite{LiPoSi}] \label{CH3}

Given  $u_0\in H^1(\mathbb R)\cap W^{1,\infty}(\mathbb R)\equiv X$ there exist $T=T(\|u_0\|_X)>0$ and a unique solution $u=u(x,t)$ of the IVP associated to the CH equation \eqref{CH0}
such that
$$
u\in \, C([-T,T]:H^1(\mathbb R))
  \cap L^{\infty}([-T,T]:W^{1,\infty}(\mathbb R)),
$$
with
$$
\sup_{[-T,T]}\|u(t)\|_X=\sup_{[-T,T]}(\|u(t)\|_{1,2}+\|u(t)\|_{1,\infty})\leq 2C\|u_0\|_X,
$$
for some universal constant $C>0$.

Moreover, given $R>0$, the map $u_0\mapsto u$, taking the data to the solution, is continuous from the ball $\{u_0\in X :\|u_0\|_X\le R\}$
into $C([-T(R),T(R)]:H^1(\R))$.
\end{theorem}

The periodic version of this theorem was previously established  in \cite{LKT}.

 The CH equation \eqref{CH0} is a member of the   \it b-family of equations \rm:
\begin{equation}
\label{eqb}
\partial_tu+u\partial_x u +\,\partial_x(1-\partial_x^2)^{-1}\Big(\,\frac{b}{2}u^2  +\frac{3-b}{2} (\partial_xu)^2\Big)=0,\;\,b\in[0,3].
\end{equation}
For $b=2$ one gets the CH equation and for  $b=3$ one obtains the Degasperis-Procesi (DP) equation \cite{DP}, the only bi-hamiltonian and integrable models in this family, see \cite{EY}, \cite{DP}.
 \begin{remark}\label{ext}
The results in Theorem \ref{CH1} and Theorem \ref{CH2} extend directly to all the members in the b-family, see \cite{Hen}. The same applies to Theorem \ref{CH3}, and its periodic version  in \cite{LKT}.
\end{remark}

Concerning the local UCP for solutions of the IVP for the CH equation \eqref{CH0}, in fact for all $b$-equations in \eqref{eqb}, we have: 
  
\begin{theorem}[\cite{LP20}] \label{CH4}  

Let   $u=u(x,t)$
be the solution of the IVP for  an equation in \eqref{eqb} provided by Theorem \ref{CH3} (see Remark \ref{ext}).
If there exist an open interval $I \subset \mathbb R$ and a time $t_1\in [0,T]$ such that 
$$
u(x,t_1)=\partial_tu(x,t_1)=0,\;\;\;\;\;\;\;x\in I,
$$
then $\,u\equiv 0$.

In particular, if there exists an open set $\Omega \subset \mathbb R\times [0,T]$ such that 
$$
u(x,t)=0\;\;\;\;\;\;(x,t)\in \Omega,
$$
then $\,u\equiv 0$.
\end{theorem}

The result in Theorem \ref{CH4} extends to solution of the IPBVP :

\begin{theorem} [\cite{LP20}]\label{CH5}
Let   $u=u(x,t)$                               
be the local solution of the IPBVP for an equation in \eqref{eqb} found in \cite{LKT}, (see Remark \ref{ext} (ii)).
If there exist an open interval $\,I\subset \mathbb S$ and a time $t_1\in[0,T]$ such that 
$$
u(x,t_1)=\partial_tu(x,t_1)=0,\;\;\;\;\;\;\;x\in I,
$$
then $\,u\equiv 0$.

In particular, if there exists an open set $\Omega \subset \mathbb R\times [0,T]$ such that 
$$
u(x,t)=0\;\;\;\;\;\;(x,t)\in \Omega,
$$
then $\,u\equiv 0$.
\end{theorem}
\begin{remark}\label{QQ}\hskip10pt

\begin{enumerate}

\item
 Others UCP for the IPBVP for the b-equations in \eqref{eqb}  have been obtained in \cite{BrCo} and \cite{Br}. 
\item
From the above results, one has that for the CH equation \eqref{CH0} (and for all the members in the b-family \eqref{eqb}) both local and asymptotic at infinity UCP's are known only in the case where the solution $u_2\equiv 0$. Therefore, the UCP result for two arbitrary solutions $u_1,\,u_2$ remains \underline{open}.

\end{enumerate}
\end{remark}

Next, we shall sketch the proof of Theorem \ref{CH4} and remark that the proof for Theorem \ref{CH5} is similar, see \cite{LP20}.

\begin{proof}  [Proof of Theorem \ref{CH4}] We observe that formally $$
  (1-\partial_x^2)^{-1}h(x)=\frac{1}{2}\,\big(e^{-|\cdot| }\ast h \big)(x),\;\;\;h\in L^2(\mathbb R).
  $$
  Thus, from the hypothesis  
  $$
  \Big(\frac{b}{2} \,u^2+\frac{3-b}{2}(\partial_xu)^2\Big)(x,t_1)=0,\;\;\;x\in I,\;\;\;b\in[0,3],
  $$
and from the equation
$$
  \partial_x(1-\partial_x^2)^{-1}\Big(\frac{b}{2} \,u^2+\frac{3-b}{2}(\partial_xu)^2\Big)(x,t_1)=0,\;\;\;x\in I.$$
  
 Thus, defining 
\begin{equation*}
\begin{split}
F(x)&\equiv \partial_x(1-\partial_x^2)^{-1}\Big(\frac{b}{2} \,u^2+\frac{3-b}{2}(\partial_xu)^2\Big)(x,t^*)\\
&=-\frac{1}{2}\,\text{sgn} (\cdot)\,e^{-|\cdot|} \ast \Big(\frac{b}{2} \,u^2+\frac{3-b}{2}(\partial_xu)^2\Big)(x,t^*)
\end{split}
\end{equation*}
  
and  $$f(x)\equiv \Big(\frac{b}{2} \,u^2+\frac{3-b}{2}(\partial_xu)^2\Big)(x,t^*)\geq 0$$
   one has that
 $$
 F\in L^1(\mathbb R)\cap C_{\bf{b}}(\mathbb R),\,\,f\in L^1(\mathbb R)\cap L^{\infty}(\mathbb R)$$
 and 
 $$
 F(x)=f(x)=0,\;x\in [\alpha,\beta].
 $$
  
Since for $\;y\notin [\alpha,\beta]$ it follows that
 $$\;\;-\text{sgn}(\beta-y)\,e^{-|\beta-y|}> -\text{sgn}(\alpha-y)\,e^{-|\alpha-y|},$$
with
$$
 f(y)\geq 0,\;\;\;\;y\in\mathbb R,$$

and 
\begin{equation*}
\begin{split}
F(\beta)&=-\frac{1}{2}\,\int_{-\infty}^{\infty}\text{sgn}(\beta-y) e^{-|\beta-y|}f(y)\\
&\geq -\ \frac{1}{2}\,\int_{-\infty}^{\infty}\text{sgn}(\alpha-y) e^{-|\alpha-y|}f(y)=F(\alpha),
\end{split}
\end{equation*}

with  
$$
\;F(\beta)=F(\alpha)\; \;\;\;\text{if and only if}\;\;\;\;f=\big(\frac{b}{2} \,u^2+\frac{3-b}{2}(\partial_xu)^2\big)\equiv 0,
$$
which yields the desired result.
\end{proof}

\begin{remark}\label{question1}
 One can ask if the result in Theorem \ref{CH4} still hold under the assumption 
 $$
 u(x,t)=k_0,\;\;\;\;\;\;\;(x,t)\in \Omega, 
 $$
 with $k_0\in\mathbb R-\{0\}$ and $\,\Omega\in \mathbb R\times[0,T]$ open set, or for some $t_1\in (0,T)$ and for some open interval $I\subset \mathbb R$ 
 $$
 \partial_tu(x,t_1)=0\;\;\;\;\text{and}\;\;\;\;u(x,t_1)=k_0,\;\;\;\;x\in I
 $$
 (see Theorem \ref{BOlocg}).
 
 In the case $b=0$ in \eqref{eqb} the previous proof provides the result. However, for other values of the parameter $b\in(0,3]$, in particular for the CH and DP equations,  the result is \underline{unknown}.
\end{remark}

Finally, we consider the generalized form of the CH (gCH) equation
\begin{equation}\label{CHk}
\partial_tu+3 u \partial_xu-\partial_t \partial_x^2 u+2k\partial_xu=2\partial_xu \partial_x^2u+ u\partial_x^3u, \;\;\;t,x,k\in\R.
\end{equation}
or its formally equivalent form
\begin{equation}
\label{intCHk}
\partial_t u+u\partial_xu+\partial_x(1-\partial_x^2)^{-1}\big(u^2+\frac{1}{2}\,\big(\partial_xu)^2+2ku\big)=0.
\end{equation}

 In \cite{Len} a systematic study on the existence and properties of traveling wave solutions for the gCH was provided. The definition  in \cite{Len} for traveling waves solutions of the gCH equation is equivalent to the following one:
 
 \begin{definition} \label{def} \hskip10pt
 
 A function $\,\phi(x-ct)$ is a traveling wave solution of the gCH equation if 
 
 \begin{itemize}
 \item[\rm(i)] $\phi\in H^1_{\rm loc}(\mathbb R) $ and 
 \vspace{3mm}
 \item[\rm(ii)] $\phi$ satisfies, in the weak sense, the equation
 \begin{equation}
 \label{weak-sol}
-c \phi_x+\phi \phi_x+(1-\partial_x^2)^{-1}\big(\phi^2+\frac{(\phi_x)^2}{2}+2k\phi\big)=0.
 \end{equation}
 \end{itemize}
\end{definition}

\begin{remark}\label{prop}\hskip10pt

\begin{enumerate}
\item One sees that if $\phi(x)$ is a traveling wave solution of \eqref{CHk}, then $-\phi(-x)$ is red also a traveling wave solution of \eqref{CHk} with $-k$ instead of $k$.
\vspace{3mm}
\item Similarly, if $u(x,t)$ is a solution of the gCH equation \eqref{CHk}, then
$$
v(x,t)=u(x-\alpha t,t)+\alpha,
$$
solves the same equation with $k-\alpha $ instead of $k$.
\end{enumerate}
\end{remark}
Following \cite{Len} and  restricting  to the setting   $c>0$ and $\,\| \phi\|_{\infty}>m,$  with 
\begin{equation}
\label{decay}
\lim_{|x|\to\infty}\phi(x)=m,
\end{equation}
one has :  peakons exist if and only if $k=m$. In this case $c =\| \phi\|_{\infty}$.

\vspace{3mm}

 The \it cuspons with decay, \rm  i.e.  $\phi$ is a weak solution of \eqref{weak-sol}, smooth on $\mathbb R-\{a\}$, having a cusp at $a\in\mathbb R$
$$
\lim_{x\uparrow a} \phi_x(x)= - \lim_{x\downarrow a}\phi_x(x)=\infty,
$$
increasing in $(-\infty, a)$, decreasing in $(a,\infty)$, and satisfying \eqref{decay}, exist if $ k< -m$. In this case $c =\| \phi\|_{\infty}$.

\vspace{3mm}

 The \it stumpons with decay, \rm i.e. $\phi$ is a continuous weak solution of \eqref{weak-sol} such that there exist $a, b, d\in\mathbb R$ with $\,a<b, \,d>0, $ such that
$$
\lim_{x\uparrow a} \phi_x(x)= - \lim_{x\downarrow b}\phi_x(x)=\infty,\;\,\,\,\,\text{with}\;\,\,\,\phi(x)=d,\;\;x\in[a,b],
$$
increasing in $(-\infty, a)$, decreasing in $(a,\infty)$, satisfying \eqref{decay}, (inserting the interval $[a,b]$ at the cusp of a cuspon).

\begin{remark}\label{CH-open}\hskip10pt
\begin{enumerate}
\item The argument given above for the proof of Theorem \ref{CH4} shows that if $m> 0$  a (positive) stumpon with decay can only exists if $k \leq 0$.
Similarly, if $m<0$ then one needs to have $k\geq 0$, see Remark \ref{prop}. In the case, $m=0$ one needs $k\neq 0$. 

From the existence result of stumpons in \cite{Len} this case provides a counter-example of the question stated in Remark \ref{question1}  for the gCH equation in \eqref{CHk} with $k\neq 0$. Thus, this question for the CH equation in \eqref{CH0} remains \underline{open}.

\item Similarly, by taking $m=0$ and two stumpon solutions with common  maximum realized  at intervals $I_1$ and $I_2$, $|I_1\cap I_2|>0$, one gets a negative answer to the local UCP for two arbitrary (weak) solutions of the gCH with $k \neq 0$, see Remarks \ref{QQ} (2) and \ref{prop}.
\end{enumerate}
\end{remark}
\section{The dispersion generalized  Benjamin-Ono equation}

 Next, we shall consider the   dispersion generalized Benjamin-Ono  (dgBO) equation 
 \begin{equation}
 \label{dgbo}
\partial_tu-\partial_x D_x^{\alpha}u+u\partial_xu=0,\;\;\;t,\,x\in \mathbb R,\;\;\alpha\in [-1,2],
\end{equation}
where
$$
D_x^{\alpha}=(-\partial_x^2)^{\alpha/2}.
$$

One has that in \eqref{dgbo} $\alpha=2$  corresponds to the KdV equation, $\alpha=1$ to the BO equation, 
$\alpha=0$ to the inviscid Burgers' equation (after a change of frame), and   $\alpha=-1$ to the so called Hilbert-Burgers equation. Also, we recall that only in the cases $\alpha\in \{1,2\}$ the model is completely integrable.

For the global and local well posedness of the IVP associated to the dgBO \eqref{dgbo} we refer to \cite{Sa79}, \cite{KPV91}, \cite{CoKeSt}, \cite{Her}, \cite{LiPiSa}, \cite{MoRi04},  \cite{MoRi06}, \cite{Gu}, \cite{HIKK}, \cite{KeMaRo}, \cite{MoPiVe}, \cite{MoSaTz}, \cite{Vent2}  and references therein.

For $\alpha\in [1,2]$ the dgBO equation possess traveling waves solutions
$$
u(x,t)=\phi_{\alpha}(x-t).
$$
In the case $\alpha\in(1,2)$ the existence (up to symmetry of the equation) of the traveling wave was established in \cite{We89}, and the uniqueness in \cite{FraLe}. In the case $\alpha=1$, the BO equation,
the uniqueness was previously obtained in \cite{AmTo}. However, no explicit formula is known for $\phi_{\alpha}$, when $\alpha \in(1,2)$. In \cite{KeMaRo}, the following upper bound for the decay of the traveling wave was deduced
\begin{equation}
\label{decaygbo}
\phi_{\alpha}(x)\leq \frac{c_{\alpha}}{(1+x^2)^{(1+\alpha)/2}},\;\;\;\;\;\;\alpha\in[1,2).
\end{equation}

Roughly speaking, the mild decay for $\alpha\in [1,2)$ is due to the non-smoothness of the symbol modeling the dispersive relation in \eqref{dgbo} $\sigma_{\alpha}(\xi)=\xi |\xi|^{\alpha}$.

Concerning the asymptotic at infinity UCP we have the following result (see the notation in \eqref{notZ}) :

\begin{theorem}  [\cite{FLP13}] \label{dgBOres}
 
Let $\alpha\in(1,2)$. Let $\,u\in C([0,T] : H^s(\mathbb R))$ be the solution of the IVP associated to the dgBO  equation \eqref{dgbo} with data $u_0$.
\begin{enumerate}
\item
If $s\geq \alpha(r+1)$, $r\in [3/2+\alpha,5/2+\alpha) $ and $u_0\in \dot Z_{s,r}$, then $u\in C([0,T] : \dot Z_{s,r})$.
\item
If $u\in C([0,T] : \dot Z_{9,(5/2+\alpha)^{-}})$ is  solution of the IVP associated to the dgBO  equation \eqref{dgbo}
and  there exist $t_1,t_2,t_3$ with $$\,0<t_1<t_2<t_3<T\,$$
 such that  $\,u(\cdot,t_j) \in Z_{9,5/2+\alpha},\,j=1,2,3,$ then $\,u\equiv 0$.
\item
There exist $u\in C([0,T] : \dot Z_{9,(5/2+\alpha)^{-}})$ non-trivial  solution of the IVP associated to the dgBO  equation \eqref{dgbo} and $t_1,t_2$ with $0<t_1<t_2<T$ 
such that $\,u(\cdot,t_j) \in Z_{9,5/2+\alpha},\,j=1,2$.

\end{enumerate}
 \end{theorem}

\begin{remark}
Roughly speaking, we observe that Theorem \ref{dgBOres} with $\alpha=1$ corresponds to Theorem \ref{BOres}, and that the gain in decay due to the stronger dispersion, i.e. $\alpha \in(1,2)$ is consistent with the result in \eqref{decaygbo} obtained in \cite{KeMaRo}.
\end{remark}

The next result regards local UCP for solutions of \eqref{dgbo}.

\begin{theorem} [\cite{KPPV}] \label{kppv}
 Let $u_1,\,u_2\in  C([0,T]:H^s(\mathbb R)) \cap C^1((0,T):H^{s-\alpha-1}(\mathbb R)),\;s>\alpha+3/2$ be strong solutions of the dgBO  equation \eqref{dgbo}   with $\alpha\in(-1,2)-\{0,1\}$.

 If there exist an open interval $I \subset \mathbb R $ and $t_1\in (0,T)$ such that
$$
\begin{cases}
\begin{aligned}
u_1(x,t_1)&=u_2(x,t_1),\;\;\;\;\,x\in I,\\
\partial_tu_1(x,t_1)&=\partial_tu_2(x,t_1),\;\;\;\;\,x\in I,
\end{aligned}
\end{cases}
$$
then,
$$u_1(x,t)=u_2(x,t),\;\;\;\;(x,t)\in \mathbb R\times [0,T].
$$

 In particular, if $\,u_1(x,t)=u_2(x,t)\,$ for $(x,t)\in\Omega\subset \mathbb R\times [0,T]$ an open set, then $\,u_1\equiv u_2$. 
\end{theorem} 

\begin{remark}
The condition on $s$ in Theorem \ref{kppv} is not optimal. 

\end{remark}

As in the case of the k-gBO equation \eqref{kgBO} and the ILW equation \eqref{ILW} we have as a consequence of Theorem \ref{kppv} and its proof:

\begin{theorem} \label{dgBOlocgg}

 Let $n\in \mathbb Z^+$ and $$u_1\in  C([0,T]:H^s(\mathbb R)) \cap C^1((0,T):H^{s-2}(\mathbb R)),\;s>\alpha +2n+3/2,$$ be a strong solution of
dgBO equation \eqref{dgbo}. If there exist  an open interval $\,I \subset \mathbb R$ and $t_1\in (0,T)$ such that
$$
\partial_tu_1(x,t_1)=Q_{2n-1}(x),\;\;\;\;\text{and}\;\;\;\;u(x,t_1)=P_n(x),\;\;\;\;\;x\in I,
$$
where $Q_{2n-1}$ and $P_n$ are polynomial with real coefficients of degree at most $2n-1$ and $n$, respectively, 
 then,
$$
Q_{2n-1}(x)\equiv P_n(x)\equiv u_1(x,t)\equiv 0,\;\;\;\;(x,t)\in \mathbb R\times [0,T].
$$

\end{theorem}

 The main idea  to prove Theorem \ref{kppv} and Theorem \ref{dgBOlocgg} is to apply the following result found in \cite{GhSaUh}, (see results in \cite{FaFe}, \cite{FaFe2}, \cite{Ru}, \cite{Yu}).
 
\begin{theorem} [\cite{GhSaUh}] \label{GSU}
Let  $ \,a\in (0,2)\,$and $\,f\in H^s(\mathbb R^n),\,s\in\mathbb R$. If there exists a non-empty open set $\,\Omega\subset \mathbb R^n$ such that
\begin{equation}
\label{hypGSU}
 (-\Delta)^{a}f(x)=f(x)=0,\;\;\;\;\;\text{in}\;\;\;\;\mathcal D'(\Omega),
 \end{equation}
 then $\,f\equiv 0$.
 
 \end{theorem}
 
\begin{remark}\hskip10pt

\begin{enumerate}

\item
Notice that the case $n=1$ and $a=1/2$  in Theorem \ref{GSU} corresponds to Lemma \ref{lemma1}.

\item
Theorem \ref{GSU} still  holds if one replaces \eqref{hypGSU} by the following hypothesis : there exist polynomials 
$ P_n,\,Q_k$ of degree $n, \,k$, respectively, such that
$$
(-\Delta)^af(x) = P_n(x) \;\;\;\;\text{and}\;\;\;f(x)=Q_k(x),\;\;\;\,x\in\Omega.
$$

\item Theorem \ref{GSU} tells us that a fractional porous medium equation of the form
$$
\partial_tu =(-\Delta)^a(u^{1+m}), \;\;\;t>0, \,x\in \mathbb R^n, \,m>0, \,a\in \mathbb R^+-2\mathbb Z,
$$
cannot have compact support solutions.
\end{enumerate}
\end{remark}

 The proof of Theorem \ref{GSU} is a consequence of the characterization of the fractional power of the Laplacian found in \cite{CaSi} : If
 $$
\begin{cases}
\begin{aligned}
&\Delta v+\frac{1-\alpha}{y}\,\partial_yv+\partial_y^2v=0,\;\;(x,y)\in\mathbb R^n\times(0,\infty),\\
& \,v(x,0)=f(x),
\end{aligned}
\end{cases}
$$
then for  $\alpha\in(0,2)$ there exists $c>0$ such that 
$$
c(-\Delta)^{a}f=-\lim_{y\downarrow 0} y^{1-2a} \partial_yv=\frac{1}{2a}\lim_{y\to 0} \frac{{v(x,y)-v(x,0)}}{y^{2a}}.
$$

Once, Theorem \ref{GSU} is available  the proof of Theorem \ref{kppv}  follows the argument given in the proof of Theorem \ref{BOloc}.

Motivated by Theorem \ref{GSU} and using an argument found in \cite{StTo} the following result was established in \cite{KPPV}:

\begin{theorem}
\label{afterGSU}
Let $\alpha_1,\alpha_2\in \mathbb R$, $\alpha_1-\alpha_2\notin 2\mathbb Z$ and $\,f\in H^{s}(\mathbb R^n),\;s\in\mathbb R$. If there exists a non-empty open set $\Theta\subset \mathbb R^n$ such that 
\begin{equation}
(1-\Delta)^{\alpha_j/2}f (x)=0,\;\;\;\;\;\text{in}\;\;\mathcal D'(\Theta)\;\,\;\;\;\text{for}\;\;j=1,2,
\end{equation}
 then $f \equiv  0$ in $H^s(\mathbb R^n)$.
 \end{theorem}

\section{The general fractional Schr\"odinger equation}

 In this section we shall consider  fractional Schr\"odinger (GFDS) equations of the form
\begin{equation}
 \label {fNLS}
 i\partial_t u +(\mathcal L_m)^{\alpha/2}u+V \,u+\left(W\ast F(|u|)\right)u+P(u,\bar{u})=0,
 \end{equation}
where $(x,t)\in \mathbb R^n\times\mathbb R$,
\begin{equation}
\label{def1}
\mathcal L_m=(-\Delta+m^2),\hskip15pt m\geq 0,
\end{equation}

\begin{equation}
\label{def2}
\begin{cases}
\begin{aligned}
&\alpha\in\mathbb R-2\mathbb Z,\hskip45pt \text{if}\hskip5pt m>0,\\
&\alpha\in (-n,\infty)-2\mathbb Z,\hskip10pt \text{if}\hskip5pt m=0, 
\end{aligned}
\end{cases}
\end{equation}
$V=V(x,t)$ representing the potential energy,  $W=W(|x|)$ the Hartree integrand, and $P(z,\bar{z})$ the nonlinearity  with $P(0,0)=0$.

Concrete examples of the model in \eqref{fNLS} arise  in several different contexts, for example:

\begin{enumerate}
\item[\rm(i)] when $m=0$, $W=P=0$, it was used in \cite{Las}  to describe particles in a class of  Levi stochastic processes,

\item[\rm(ii)]  when $m>0$, $W=P=0$, it was derived as a generalized semi-relativistic (Schr\"odinger) equation, see \cite{Les} and references therein,

\item[\rm(iii)]  when $m=0$, $\alpha=1$, $V=W=0$, and $P(u,\bar{u})=\pm|u|^{\alpha}u,\,\alpha>0,$ it is known as the half-wave equation, see \cite{BGLV}, \cite{GLPR} and references therein, 

\item[\rm(iv)]  when $m=1$, $V=P=0$, $F(|z|)=|z|^2$ the model arises in gravitational collapse, see \cite{ElSc}, \cite{FrLe} and references therein,

\item[\rm(v)] when $m=0$, $V=W=0$ and 
\begin{equation*}
\hskip35pt P(u,\bar{u})=c_0|z|^2z+c_1z^3+c_2z\,\bar{z}+c_3\bar{z}^3,\;\;\;\;c_0\in\mathbb R,\;c_1, c_2, c_3\in\mathbb C,
\end{equation*}
it was deduced in \cite{IoPu} in the study the long-time behavior of solutions to the water waves equations in $\mathbb R^2$, 
where $(-\partial_x^2)^{1/2}$ modeling the dispersion relation of the linearized gravity water waves equations for one-dimensional interfaces.

The well-posedness of the IVP associated to some equations of the type in \eqref{fNLS} has motivated several works, we refer to \cite{Les}, \cite{BoRi}, \cite{CHHO}, \cite{ChOz}, \cite{HoSi}, \cite{KLR} and references therein. Here, in order to simplify the exposition, we will assume that the IVP associated to the equation \eqref{fNLS} is locally well-posed in $H^s(\mathbb R^n)$ for $s\geq s^*=s^*(\alpha,n, V,W,F,P)$.
\end{enumerate}

Concerning local UCP for the IVP associated to the equation \eqref{fNLS} with $m>0$, we have the following result:

\begin{theorem} [\cite{KPPV}] \label{kppv1}\hskip10pt
\begin{enumerate}
\item  Let $\,\alpha\in (-2,2)-\{0\},\;m>0$. Let $$u_1,\,u_2\in  C([0,T]:H^{s}(\mathbb R)) \cap C^1((0,T):H^{s-\alpha}(\mathbb R)),\;s>s^*,$$ be 
two solutions of the IVP associated to the GFDS equation  \eqref{fNLS} with $W\equiv 0$. If there exist an open set $D\subset \mathbb R^n$ and $t_1\in[0,T]$ such that
\begin{equation}
\label {hyp-1}
\begin{cases}
\begin{aligned}
u_1(x,t_1)&=u_2(x,t_1),\;\;\;\;\;\;\;\;x\in D,\\
\partial_tu_1(x,t_1)&=\partial_tu_2(x,t_1),\;\;\;\;\;x\in D,
\end{aligned}
\end{cases}
\end{equation}
then
$$
u_1(x,t_1)=u_2(x,t_1),\;\;\;\;(x,t)\in\mathbb R^n\times [0,T].
$$
In particular, if $u_1(x,t)=u_2(x,t)$ for $(x,t)\in\Omega \subset\mathbb R^n\times[0,T]$ with $\Omega$ open, then $u_1\equiv u_2$.

\item For the general form of the GFDS equation \eqref{fNLS}, i.e. $W\neq 0$, the result in (1) still holds if  $u_2(x,t)=0$ for $(x,t)\in \mathbb R^n\times [0,T]$.
\end{enumerate}
\end{theorem}

Next, we consider the case $m=0$:

\begin{theorem}[\cite{KPPV}] \label{kppv2}\hskip10pt
\begin{enumerate}
\item Let $\,\alpha\in (0,2)$ and $m=0$. Let $\,u_1,\,u_2$ be two solutions of the IVP associated to the equation  \eqref{fNLS} with $W\equiv 0$ such that
\begin{equation}
\label{class2m}
u_1,\,u_2\in C([0,T]:H^{s}(\mathbb R^n))\cap C^1([0,T]:H^{s-\alpha}(\mathbb R^n)),
\end{equation}
with $s>\max\{\alpha; n/2\}$. If there exist an open set $D\subset \mathbb R^n$ and $t_1\in[0,T]$ such that
\begin{equation}
\label {hyp-2}
\begin{cases}
\begin{aligned}
u_1(x,t_1)&=u_2(x,t_1),\;\;\;\;\;\;\;\;x\in D,\\
\partial_tu_1(x,t_1)&=\partial_tu_2(x,t_1),\;\;\;\;\;x\in D,
\end{aligned}
\end{cases}
\end{equation}
then
$$
u_1(x,t_1)=u_2(x,t_1),\;\;\;\;(x,t)\in\mathbb R^n\times [0,T].
$$
In particular, if $u_1(x,t)=u_2(x,t)$ for $(x,t)\in\Omega \subset\mathbb R^n\times[0,T]$ with $\Omega$ open, then $u_1\equiv u_2$.

\item For the general form of  the equation \eqref{fNLS}, i.e. $W\not\equiv 0$, the result in (1) still holds if one assumes that $u_2(x,t)=0$ for $(x,t)\in \mathbb R^n\times [0,T]$.
\end{enumerate}
\end{theorem}

\begin{remark}\hskip10pt

\begin{enumerate}

\item In the case  $V\equiv W\equiv 0$ the appropriate versions of Theorem \ref{BOlocg}, Theorem \ref{ILWlocg}, and Theorem \ref{dgBOlocgg} still hold for the IVP associated 
to the equation \eqref{fNLS}.

\item The proofs of Theorem \ref{kppv1} and Theorem \ref{kppv2} are based on the results in Theorem \ref{afterGSU} and Theorem \ref{GSU}, respectively (see \cite{KPPV}).

\end{enumerate}
\end{remark}

Finally, we briefly discuss the asymptotic UCP for solution of the IVP associated to the equation \eqref{fNLS}. This question is largely \underline{open}, even in the case when the operator modeling the dispersion $\mathcal L_m$ relation is local, i.e. $\alpha \in 2\mathbb Z^+$ and $m\geq 0$. 

Let us consider some particular cases. First, we fix $\alpha=2$ and $m=0$, so $(\mathcal L_m)^{\alpha/2}=-\Delta$. In this setting, for the associated linear problem 
with $V\equiv W\equiv P\equiv 0$ the asymptotic UCP can be re-phrased in terms of uncertainty principles for the Fourier transform. For example, the $L^2$ version of Hardy uncertainty principle \cite{Har} found in \cite{CoPr} can be stated as:
$$
\text{If}\;f\,e^{|x|^2/\beta^2},\,\widehat{f} \,e^{|x|^2/\alpha^2}\in L^2(\mathbb R^n)\;\;\text{and}\;\;1/\alpha\beta\geq 1/4\Rightarrow f\equiv 0.
$$

This result can be reformulated in terms of the free Schr\"odinger group $\{e^{it\Delta}\,:\,t\in\mathbb R\}$ : 
\begin{equation}
\label{hardy}
\text{If}\,f e^{|x|^2/\beta^2},\,e^{iT\Delta}f\,e^{|x|^2/\alpha^2}\in L^2(\mathbb R^n)\;\text{and}\;T/\alpha\beta\geq 1/4\Rightarrow f\equiv 0.
\end{equation}
We observe that in this case the weight at time $t=0$ and $t=T$ may be different. This is a general fact exploited in several works, see \cite{EKPV12} and references therein.

In \cite{EKPV10}, the result in \eqref{hardy}, except for the case $T/\alpha \beta=1/4$, was enlarged for solutions $u(x,t)$ of the equation
\begin{equation}
\label{pot-case}
\partial_tu=i(\Delta u+V(x,t)u),\;\;\;\;\;\;(x,t)\in\mathbb R^n\times[0,T],
\end{equation}
with the complex valued potential $V=V(x,t)$ satisfying the decay assumption
\begin{equation}
\label{decay-p}
\lim_{R\uparrow \infty} \| V\|_{L^1([0,T]:L^{\infty}(\mathbb R^n-B_R(0)))}=0.
\end{equation}
Moreover, in \cite{EKPV10} an example of a complex valued potential $V(x,t)$ verifying \eqref{pot-case} for which the corresponding solution $u(x,t)\not\equiv 0$ holds that
$$
u(x,0)\,e^{|x|^2/\beta^2},\,u(x,T)\,e^{|x|^2/\alpha^2}\in L^2(\mathbb R^n)\;\;\text{with}\;\;T/\alpha\beta= 1/4,
$$
was given. This affirms that the result is sharp.

A similar result for real valued potential $V(x,t)$ is \underline{open}.

The result in \cite{EKPV10} commented above extends to the nonlinear equation
\begin{equation}
\label{NLS}
\partial_tu=i\Delta u+iV(x,t)u +iP(u,\overline{u}),\;\;\;\;\;\;(x,t)\in\mathbb R^n\times[0,T],
\end{equation}
with the potential $V$ satisfying \eqref{decay-p},  $ P :\mathbb C^2\to \mathbb C$ being a smooth function with 
$$
P(0,0)=\partial_zP(0,0)=\partial_{\overline{z}}P(0,0)=0,
$$
and $u_1, \,u_2 $ are regular enough solutions of the equation \eqref{NLS} such that 
$$
(u_1-u_2)(x,0)e^{|x|^2/\beta^2},\,(u_1-u_2)(x,T)e^{|x|^2/\alpha^2}\in L^2(\mathbb R^n),\,\,\,T/\alpha\beta> 1/4,
$$
then $u_1\equiv u_2$. 

However, this result may not be optimal. Consider the case of \eqref{NLS} with $V\equiv 0$ and nonlinearity $ P(z,\overline{z})= |z|^{\alpha-1}z$ and $\alpha\in (1,\infty)$. Assuming $u_2\equiv 0$ one may ask the question : what is the  strongest possible decay at infinity of  a non-trivial solution $u_1(x,t)$ of 
\begin{equation}
\label{NLS1}
\partial_tu=i(\Delta u+ |u|^{\alpha-1}u),\, 
\end{equation}
at the time $t=0$ and $t=1$? The answer is  \underline{unknown}.

On one hand, the result commented above, for $\alpha $ odd or sufficiently large, shows that 
\begin{equation}
\label{100}
\text{if}\;\;u_1(x,0)\,e^{|x|^2/4},\,u_1(x,1)\,e^{|x|^2/4}\in L^2(\mathbb R^n)\;\;\;\text{then}\;\;\;u_1\equiv 0.
\end{equation}

On the other hand, one has that the ground state solution 
$$
u_1(x,t)=e^{i t}\varphi(x)
$$  
of \eqref{NLS1}  where $\varphi\in H^1(\mathbb R^n)$ is the unique (up to symmetries) radial positive solution of the elliptic problem
$$
-\Delta \varphi+\varphi=|\varphi|^{\alpha-1}\varphi
$$
with $\varphi$ having exponential decay at infinity, i.e. 
\begin{equation}
\label{decayNLS}
\varphi(x)=O(e^{-a|x|}), \;\;\;\;\text{as}\;\;\;\;\; |x|\uparrow \infty,\;\;\;a\in\mathbb R,
\end{equation}
 see \cite{Str} and \cite{BLi}. 

We \underline{do not know} a result as that in \eqref{100} with a weaker weight that the Gaussian $e^{|x|^2/4}$ at time $t=0$ and $t=1$. Also, we  \underline{do not know} a solution of the equation \eqref{NLS1} which has a stronger decay that $u_1(x,t)=e^{i t}\varphi(x)$ for $t\in [0,1]$, see \eqref{decayNLS}.

Returning to the general equation in \eqref{fNLS} one observes that for $m=0$ the symbol of the dispersive operator is non-smooth. Therefore, from the discussed results of the BO equation one expect the strongest possible decay of the solution to be weaker than that for the case $m>0$ where the symbol is smooth. Also, the form of the Hartree integrand should play an essential role in the strongest possible decay of the solution.

Finally, we have some comments concerning the product rule for fractional derivatives. Theorem \ref{GSU} tells us that the pointwise inequality involving the fractional Riesz potential $D=(-\Delta)^{1/2}$:
\begin{equation}
\label{LR1}
|D^a(fg)(x)| \leq c\,\big(|(f D^a g)(x)+|(g D^a f)(x)|\big) \;\;\;\text{a.e.}\;\;\mathbb R^n
\end{equation}
fails for any $a\in(0,2) $ and any constant $c>0$. To see this, one takes $f, g\in C^{\infty}_0(\mathbb R^n)$ with $|\supp(f)\cap  \supp(g)|>0$, to have the right hand side of \eqref{LR1} vanishing in $(\supp(f)\cup \supp(g))^c$ with the left hand side of \eqref{LR1} non-vanishing in any open set of 
$(\supp(f))^c\cup (\supp(g))^c$.

In this regard the following pointwise estimate was established in \cite{CoCo} :

\begin{theorem}
[\cite{CoCo}]\label{coco1}
Let $a\in (0,2)$ and any $f\in C^{\infty}_0(\mathbb R^n)$. Then the following inequality holds:  
\begin{equation}
\label{coco2}
D^a(f^2)(x) \leq 2 f(x) D^af(x).
\end{equation}

\end{theorem}
For a generalization of Theorem \ref{coco1} we refer to \cite{CaSir}.

From Theorem \ref{GSU} and Theorem \ref{coco1} one has : if $\,f\in C^{\infty}_0(\mathbb R^n)$, then 
$$
D^a(f^2)(x) \leq 0\;\;\;\;\;\;\;\forall x \,\in (\supp(f))^c,
$$
with $D^a(f^2)$ non-vanishing in any open subset of $ (\supp(f))^c$.

\vskip.1in
However, one has the following estimate:
\begin{theorem}
\label{tlr}
Let $r\in [1,\infty]$  and $ p_1,p_2, q_1, q_2\in (1,\infty]$ with 
\begin{equation}
\label{LR4}
1/r=1/p_1+1/q_1=1/p_2+1/q_2.
\end{equation}
 Given $a>0$ there exists $c=c(n,a,r,p_1,p_2,q_1,q_2)>0$ such that for all $f, g\in \mathcal S(\mathbb R^n)$
one has
\begin{equation}
\label{LR2}
\| D^a(fg)\|_r\leq c\big(\|f\|_{p_1}\| D^ag\|_{q_1}+\|g\|_{p_2}\|D^af\|_{q_2}\big).
\end{equation}
\end{theorem}

For the proof of Theorem \ref{tlr} we refer to \cite{GO}. The case $r=p_1=p_2=q_1=q_2=\infty$ was established in \cite{BoLi}, see also \cite{GMN}. For earlier versions of this result we refer to \cite{KaPo} and \cite{KPV93}.

\medskip

 From \eqref{LR1} and \eqref{LR2} it seems natural to \underline{ask whether or not} the following weaker estimate than \eqref{LR2} still holds:
 
Given $p\!\in\!\! [1,\infty]$ there exists $c\!=\!c(n,p)\!\!>\!0$ such that  for all $f, g\!\in\! \mathcal S(\mathbb R^n)$
\begin{equation}
\label{LR5}
\| D^a(fg)\|_p\leq c\big(\|f D^ag\|_{p}+\|g D^af\|_{p}\big).
\end{equation}

\end{document}